\documentclass[12pt,reqno]{amsart}
\usepackage{hyperref}
\usepackage{soul}
\usepackage{ amssymb }
\numberwithin{equation}{section}
\usepackage{hyperref}
\usepackage{float}
\usepackage{pgf,tikz}

\newcommand \supp{\operatorname{supp}}

\newcommand \cochord{\operatorname{co-chord}}
\newcommand \reg{\operatorname{reg}}

\newcommand \Tor{\operatorname{Tor}}

\newcommand \ma{\operatorname{m}}

\newcommand \MM{\operatorname{min-match}}

\newcommand\PP{\mathcal{P}}

\newcommand \FA{\mathcal{F}}

\newcommand \K{\mathbb{K}}

\newtheorem{theorem}{Theorem}[section]
\newtheorem{definition}[theorem]{Definition}

\newtheorem{lemma}[theorem]{Lemma}

\newtheorem{example}[theorem]{Example}
\newtheorem{obs}[theorem]{Observation}
\newtheorem{question}[theorem]{Question}
\newtheorem{remark}[theorem]{Remark}

\newtheorem{dis}[theorem]{Discussion}
\newtheorem{corollary}[theorem]{Corollary}
\newtheorem{setup}[theorem]{Set-up}

\newtheorem*{notation*}{Notation}
\newtheorem{notation}[theorem]{Notation}

\begin{document}

\title{Bounds for the regularity of product of edge ideals}

\author[A.~Banerjee]{Arindam Banerjee}

\address{Department of Mathematics, 
Indian Institute of Technology Kharagpur, 721302, India}
\email{123.arindam@gmail.com}
\author[P. Das]{Priya Das}
\address{Department of Mathematics, National Institute of Technology, Calicut, Kerala-673601, India}

\email{priya.math88@gmail.com}
\author[S ~Selvaraja]{S ~Selvaraja}
\address{Chennai Mathematical Institute, H1, SIPCOT IT Park, Siruseri, Kelambakkam, Chennai, INDIA - 603103.}
\email{selva.y2s@gmail.com, sselvaraja@cmi.ac.in}

\thanks{AMS Classification 2010. Primary: 13D02, 05E45, 05C70}
\maketitle
\begin{abstract}
 Let $I$ and $J$ be edge ideals in  a polynomial ring  $R = \K[x_1,\ldots,x_n]$ with $I \subseteq J$. 
 In this paper, we obtain a general upper and lower bound for the  Castelnuovo-Mumford regularity 
 of $IJ$ in terms of certain invariants associated with $I$ and $J$. 
 Using these results, we explicitly compute the regularity of $IJ$ for several classes of edge ideals.
 In particular, we compute the regularity of $IJ$ when $J$ has linear resolution.
 Finally, we compute the precise expression for the regularity
 of $J_1 J_2\cdots J_d$, $d \in \{3,4\}$, 
 where $J_1,\ldots,J_d$ are edge ideals,
 $J_1 \subseteq J_2 \subseteq \cdots \subseteq J_d$ and 
 $J_d$ is the edge ideal of a complete graph.
\end{abstract}

\section{Introduction}
  Let $M$ be a finitely generated graded module over $R = \K[x_1,\ldots,x_n]$ where
$\K$ is a field. 
The Castelnuovo-Mumford regularity (or simply, regularity) of $M$, denoted by $\reg(M)$, is defined 
to be the least integer $i$ so that, for every $j$, the $j^{\text{th}}$ syzygy of $M$
is generated in degrees $\leq i+j$.
Regularity  is an important invariant
in commutative algebra and algebraic geometry that measures the computational
complexity of ideals, modules, and sheaves. 
In this paper, we study bounds on the regularity of product of ideals in a 
 polynomial ring. 
 
 The regularity of products of ideals was studied first by Conca and Herzog \cite{CH03}.  
 They studied whether for homogeneous ideal $I$ and finitely 
 generated graded module $M$ over $R$, one has 
 $\reg(IM) \leq \reg(I) + \reg(M).$ 
 This question is essentially a generalization of the simple fact that the highest degree of a 
 generator of the product $IM$ is bounded above by the sum of the highest degree of a generator 
 of $M$ and the highest degree of a generator of $I$. The answer to this question is negative 
 in general. There are several examples already  known with $M=I$ such that 
 $\reg(I^2)>2\reg(I)$,
 see Sturmfels \cite{Stu00}.  
 They found some special classes of ideal $I$ and module $M$ such that $\reg(IM) \leq \reg(I)+\reg(M)$.
 In particular, they showed that  if $I$ is a homogeneous ideal in a 
 polynomial  ring $R$ with $\dim(R/I)\leq 1$, then 
 $\reg(IM) \leq \reg(I)+ \reg(M)$ 
 for any finitely generated module 
 $M$ over $R$.  

In case $M$ is also a homogeneous ideal, the situation becomes particularly interesting. 
For example, Sidman proved that if $\dim(R/(I+J)) \leq 1$,
then the regularity of $IJ$ is bounded above by $\reg(I) + \reg(J)$, \cite{Sid02}.
Also, she proved that if two ideals of $R$, say $I$ and $J$, define schemes whose 
intersection is a finite set of points,
then $\reg(IJ)\leq \reg(I)+\reg(J)$. 
In \cite{CMT07}, Chardin, Minh and Trung proved that if 
$I$ and $J$ are monomial complete intersections, then $\reg(IJ) \leq \reg(I)+\reg(J)$. 
Cimpoea\c{s} proved that for two monomial ideals of Borel type $I,J$, $\reg(IJ) \leq \reg(I)+\reg(J)$,
\cite{Cim09}.
Caviglia in \cite{Cav07} and Eisenbud, Huneke and Ulrich in \cite{EHU06}
studied the more general problem of the regularity of tensor products and various $\Tor$ modules 
of $R/I$ and $R/J$.

 In this paper, we study the same problem for the case of edge ideals and seek for better bounds by 
  exploiting the combinatorics of the underlying graph. 
  Let $G$ be a finite simple graph without isolated vertices  on the vertex set $\{x_1,\ldots,x_n\}$ and $I(G) := (\{x_ix_j \mid
\{x_i, x_j\} \in E(G)\}) \subset R=\K[x_1,\ldots,x_n]$ be the \textit{edge ideal} corresponding to the graph $G$.
In general, computing the regularity of $I(G)$ is NP-hard  
(\cite[Corollary 23]{russ}).
Several recent papers have related the $\reg(I(G))$ with various
invariants of the graph $G$ (see \cite{BBH17} for
a survey in this direction). A primary inspiration for this paper is
Katzman's and Woodroofe's theorem from \cite{kat} and \cite{russ}. They showed that if $G$ is a graph, then 
\begin{equation}\label{ag_reg_chg}
\nu(G) + 1 \leq \reg(I(G)) \leq\cochord(G) + 1, 
\end{equation}
where $\nu(G)$ denotes the induced matching number of
$G$ (see Section \ref{pre} for the definition) and 
$\cochord(G)$ denotes the co-chordal cover number of $G$ (see Section \ref{pre} for the definition).
In this context, the natural question arises if 
$I$ and $J$ are edge ideals in $R$,
then what is the regularity of $IJ$? 
This question give rise to two directions of research. One direction is to obtain
the precise expression for $\reg(IJ)$ for particular classes of edge ideals. 
Another direction is to obtain upper
and lower bounds for $\reg(IJ)$ using combinatorial invariants associated to the graphs.
Therefore, one may ask if
$I$ and $J$ are edge ideals, then 
\begin{enumerate}
 \item[Q1] what are the lower and upper bounds for the regularity of $IJ$ using combinatorial 
invariants associated to the graphs?
\item[Q2] what is the precise expression for the regularity of $IJ$ 
for particular classes of graphs?
\end{enumerate}

This paper evolves around these two questions.

Computing the regularity of product of two edge ideals of graphs seems more challenging 
compared to the regularity of edge ideal of a graph. Even in the case of simple classes of graphs, 
the regularity of product of two edge ideals is not known. 
So, naturally one restricts the attention to important subclasses.
We are therefore interested in families of edge ideals $I$ and $J$ with $I \subseteq J$.

First, we  prove the lower bound for the regularity of product of more than two edge ideals. More precisely,
let $J_i=I(G_i)$ be the edge ideal of $G_i$ and $J_1 \subseteq \cdots \subseteq J_d$ 
for all $1 \leq i \leq d$.
  Then $2d+\nu_{G_1 \cdots G_d}-1 \leq \reg(J_1 \cdots J_d)$,
  where $\nu_{G_1 \cdots G_d}$ denotes the joint induced matching number of $G_i$ 
  (see Section \ref{pre} for the definition) for all
  $1 \leq i \leq d$
  (Theorem \ref{main-lower}). 
We prove an upper bound for the regularity of product of two edge
  ideals in terms of 
co-chordal cover numbers. We prove that if
$G$ is a graph and $H$ is a subgraph of $G$ with $I=I(H)$ and $J=I(G)$, then
 $\reg(IJ) \leq \max \{\cochord(G)+3,~\reg(I)\}.$
 In particular,
$\reg(IJ) \leq \max \{\cochord(G)+3,~\cochord(H)+1\}$ (Theorem \ref{main}).
The above bound is inspired by the general upper bound for the regularity of powers
of edge ideals given in \cite[Theorem 3.6]{jayanthan} and \cite[Theorem 4.4]{JS21}.
Theorem \ref{main} has a number of interesting consequences. For example, 
Corollary \ref{mat-bound}, says that if $H$ is any subgraph of $G$,
then $\reg(IJ) \leq \ma(G)+3$ where $\ma(G)$ denotes  the matching number of $G$.
On the other hand, Corollary \ref{ind-eql}, says that 
if $H$ is an induced subgraph of $G$, then 
 $\nu(H)+3 \leq \reg(IJ) \leq \cochord(G)+3.$

 We then move on to compute the precise expression for the regularity of product of edge ideals.
 As  a consequence of the techniques that we have developed, we explicitly compute the regularity of $IJ$ 
when $J$ has linear resolution (Theorem \ref{regularity}). 
Next, we study the regularity of product of more than two edge ideals.
We compute the precise expression for $\reg(J_1\cdots J_d)$ when $J_1 \subseteq \cdots \subseteq J_d$, $d \in \{3,4\}$ and 
$J_d$ is the edge ideal of complete graph (Theorem \ref{prd-ideals}).
We now use Theorem \ref{main} and Theorem \ref{prd-ideals} to get an upper bound for the regularity of
$J_1 \cdots J_d$ in terms of co-chordal cover numbers (Corollary \ref{prd-us}).
As an immediate consequence of above results, we give sufficient conditions for product of edge ideals
to have linear resolutions (Corollary \ref{se-cl}, Corollary \ref{ls-pr}).

Our paper is organized as follows. In Section \ref{pre}, we collect the necessary notions, 
  terminologies and some results that are used subsequently.
  We prove, in Section \ref{tech}, several
technical lemmas which are needed for the proof of our main results which appear in
Sections \ref{2-edgeideals} and \ref{precise}.

\section{Preliminaries}\label{pre}
In this section, we set up the basic definitions and notation needed for the main results. Let $G$ be a finite simple graph with vertex set $V(G)$
and edge set $E(G)$. 
 A subgraph $L \subseteq G$  is called \textit{induced} if $\{u,v\}$ is an edge of $L$ 
if and only if $u$ and $v$ are vertices of $L$ and $\{u,v\}$ is an edge of $G$.
For $\{u_1,\ldots,u_r\}  \subseteq V(G)$, let $N_G(u_1,\ldots,u_r) = \{v \in V (G)\mid \{u_i, v\} \in E(G)~ 
\text{for some $1 \leq i \leq r$}\}$
and $N_G[u_1,\ldots,u_r]= N_G(u_1,\ldots,u_r) \cup \{u_1,\ldots,u_r\}$. 
For $U \subseteq V(G)$, we denote by $G \setminus U$ 
the induced subgraph of $G$ on the vertex set $V(G) \setminus U$.
Let $C_k$ denote the cycle on $k$ vertices.

Let $G$ be a graph. We say $2$ non-adjacent edges
  $\{f_1,f_2\}$ form an $2K_2$ in $G$ if $G$ does not have an edge with one 
endpoint in $f_1$ and the other in $f_2$. A graph without $2K_2$ is called $2K_2$-free
also called \textit{gap-free} graph.


A \textit{matching} in a graph $G$ is a subgraph consisting of pairwise disjoint edges.
The \textit{matching number} of $G$, denoted by $\ma(G)$, is the
maximum cardinality among matchings of $G$. 
If the subgraph is an induced subgraph, the matching is an \textit{induced matching}.
The largest size of an induced matching in $G$ is called its \textit{induced 
matching number} and denoted by $\nu(G)$.  
The \textit{complement} of a graph $G$, denoted by $G^c$, is the graph on the same
vertex set in which $\{u,v\}$ is an edge of $G^c$ if and only if it is not an edge of $G$.
A graph $G$ is \textit{chordal} if every induced
cycle in $G$ has length $3$, and is co-chordal if $G^c$ is chordal.
The \textit{co-chordal cover number}, denoted $\cochord(G)$, is the
minimum number $n$ such that there exist co-chordal subgraphs
$H_1,\ldots, H_n$ of $G$ with $E(G) = \bigcup_{i=1}^n E(H_i)$.

Let $G_i$ is a graph, for all $1 \leq i \leq d$, and $G_i$ be a subgraph
of $G_{i+1}$ for all $1 \leq i \leq d-1$.
The largest size of an induced matching in $G_i$ for all $1 \leq i \leq d$ is called  the \textit{joint induced 
matching number} and denoted by $\nu_{G_1 \cdots G_d}$.
Note that if $G_i$ is an induced subgraph of $G_{i+1}$ for all
$1 \leq i \leq d-1$, then $\nu_{G_1 \cdots G_d}=\nu(G_1)$.

\begin{example}
 Let $G$ be the graph as shown in figure. 
 Then $\{ \{x_1,x_2\}$, $\{x_3,x_4\}$, $\{x_5,x_6\}$, $\{x_7,x_8\}\}$
 forms a matching of $G$, but not an induced matching. 
 The set 
 $\{ \{x_1,x_2\}, \{x_4,x_5\}\}$ forms an induced matching.
 Then $\nu(G) \geq 2$. It is not hard to verify that $\nu(G)=2$.
 Let $H$ be a subgraph of $G$ with $E(H)=\{ \{x_1,x_2\}, \{x_3,x_4\}\}$.
  Since $H$ is a disjoint union two edges, $\nu(H)=2$. The set $\{\{x_1,x_2\}\}$ forms an 
  induced matching of $G$ and $H$. Then $\nu_{HG} \geq 1$. Since the set 
  $\{\{x_1,x_2\},\{x_3,x_4\}\}$ forms an induced matching of $H$ but not in $G$, 
  $\nu_{HG} =1$.

 \begin{minipage}{\linewidth}
\begin{minipage}{0.23\linewidth}
\begin{figure}[H]

\begin{tikzpicture}[scale=.64]
\clip(1.42,1.6) rectangle (13.66,5.6);
\draw [line width=1.pt] (4.,5.)-- (5.,5.);
\draw [line width=1.pt] (4.,5.)-- (3.,4.);
\draw [line width=1.pt] (3.,4.)-- (3.,3.);
\draw [line width=1.pt] (3.,3.)-- (4.,2.);
\draw [line width=1.pt] (4.,2.)-- (5.,2.);
\draw [line width=1.pt] (5.96,2.86)-- (5.,2.);
\draw [line width=1.pt] (6.,4.)-- (5.96,2.86);
\draw [line width=1.pt] (5.,5.)-- (6.,4.);
\begin{scriptsize}
\draw [fill=black] (4.,5.) circle (2.0pt);
\draw[color=black] (4.19,5.38) node {$x_8$};
\draw [fill=black] (5.,5.) circle (2.0pt);
\draw[color=black] (5.19,5.38) node {$x_1$};
\draw [fill=black] (3.,4.) circle (2.0pt);
\draw[color=black] (2.57,4.14) node {$x_7$};
\draw [fill=black] (3.,3.) circle (2.0pt);
\draw[color=black] (2.53,3.06) node {$x_6$};
\draw [fill=black] (4.,2.) circle (2.0pt);
\draw[color=black] (3.95,1.76) node {$x_5$};
\draw [fill=black] (5.,2.) circle (2.0pt);
\draw[color=black] (5.35,2.0) node {$x_4$};
\draw [fill=black] (5.96,2.86) circle (2.0pt);
\draw[color=black] (6.35,3.04) node {$x_3$};
\draw [fill=black] (6.,4.) circle (2.0pt);
\draw[color=black] (6.19,4.38) node {$x_2$};
\end{scriptsize}
\end{tikzpicture}
\end{figure}
\end{minipage}
\begin{minipage}{0.75\linewidth}
 Let $H_1$, $H_2$ and $H_3$ be the subgraphs of $G$ with 
$E(H_1)=\{\{x_1,x_2\},\{x_2,x_3\},\{x_3,x_4\} \}$,
$E(H_2)=\{\{x_4,x_5\},\{x_5,x_6\},\{x_6,x_7\} \}$ and
$E(H_3)=\{\{x_7,x_8\},\{x_8,x_1\}\}$ respectively.
We can seen that $H_1$, $H_2$ and $H_3$ are
co-chordal subgraphs of $G$ and $E(G)=\bigcup_{i=1}^3E(H_i)$.
Therefore, $\cochord(G) \leq 3$. It is also  
not hard to verify that $\cochord(G)=3$. 
\end{minipage}
\end{minipage}
\end{example}

Polarization is a process to obtain a squarefree monomial ideal from a given monomial ideal. 
\begin{definition}\label{pol_def}
 Let $M=x_1^{a_1}\cdots x_n^{a_n}$ be a monomial in  $R=\K[x_1,\dots,x_n]$. 
Then we define the squarefree monomial $P(M)$ ({\it polarization} of $M$) as 
$$P(M)=x_{11}\cdots x_{1a_1}x_{21}\cdots x_{2a_2}\cdots x_{n1}\cdots x_{na_n}$$ in the polynomial ring 
$R_1=\K[x_{ij} \mid 1\leq i\leq n,1\leq j\leq a_i]$.
If $I=(M_1,\dots,M_q)$ is an ideal in $R$, then the polarization of
$I$, denoted by $\widetilde{I}$, is defined as $\widetilde{I}=(P(M_1),\dots,P(M_q))$.
\end{definition}

Let $M$ be a graded $R = \K[x_1,\ldots,x_n]$ module. For non-negative
integers $i, j$, let $\beta_{i,j}(M)$ denote the $(i,j)$-th graded Betti
number of $M$.
In this paper, we repeatedly use one of the important properties of the
polarization, namely:
\begin{corollary} \cite[Corollary 1.6.3(a)]{Herzog'sBook}\label{pol_reg} 
Let $I \subseteq R=\K[x_1,\ldots,x_n]$ be a monomial ideal. 
 If $\widetilde{I} \subseteq \widetilde{R}$ is a polarization of $I$, then for 
 all $i,j$, $\beta_{i,j}(R/I)=\beta_{i,j}(\widetilde{R} / \widetilde {I})$. 
 In particular, $\reg(R/I)=\reg (\widetilde R/\widetilde I)$.
\end{corollary}

\section{Technical lemmas}\label{tech}

In this section we prove several technical results concerning the graph associated with
$\widetilde{(IJ:ab)}$, for any $ab \in I$, where $I$ and $J$ are edge ideals and $I \subseteq J$.
We first fix the set-up that we consider throughout this paper.
\begin{setup}\label{setup}
Let $G$ be a graph and $H$ be a subgraph of $G$. Set $I=I(H)$ and $J=I(G)$.
 \end{setup}
For a monomial ideal $K$, let $\mathcal{G}(K)$ denote the minimal generating set of $K$.
For a monomial $m \in R = \K[x_1,\ldots,x_n]$, support of $m$ is the set of variables appearing in
$m$ and is denoted by $\supp(m)$, i.e., $\supp(m) = \{x_i \mid x_i \text{ divides } m \}$.

The following result is being used repeatedly in this paper:

\begin{lemma} \label{stu-lemma}
Let $I$ and $J$ be as in Set-up \ref{setup}.
  Then the colon ideal
 $(IJ:ab)$ is a generated by
 quadratic monomial ideal for any $ab \in I$. 
 More precisely,
 \[
  (IJ:ab)=J+ K_1 + K_2,
 \]
where 
$K_1=(pq \mid p \in N_G(a) \text{ and } q \in N_H(b))$ and 
$K_2=(rs \mid r \in N_H(a) \text{ and } s\in N_G(b))$.
\end{lemma}
\begin{proof}


Let $m \in \mathcal{G}((IJ:ab))$. 
By degree consideration $m$ can not have degree 1. Suppose $\deg(m) \geq 3$.
Then there exists $e \in \mathcal{G}(I)$ and 
$f \in \mathcal{G}(J)$ such that $ef\mid mab$. Since $m$ is a minimal monomial generated of 
$(IJ:ab)$, there does not 
exist $m'$,  $m' \neq m$ and $m' \mid m$ such that $ef\mid m'ab$.  
If there exist $g \in \mathcal{G}(J)$ such that $g\mid m$, then for minimality 
of $m$ and $g \in (IJ:ab)$ both implies $g=m$.
This is a contradiction to $\deg(m) \geq 3$.
Therefore, $\deg(m)=2$. We assume that $g \nmid m$ for any $g \in \mathcal{G}(J)$.
Then $e \nmid ab$. Let $e=ax$, where $x\mid m$. Therefore, 
$xf\mid mb$. 
If $f=by$, where $y\mid (\frac{m}{x})$, then $xy\mid m$. Hence, 
by minimality of $m$, $m$ is a quadratic monomial. 
Similarly, for $e=bx$ we can prove in a similar manner.

 Clearly, $J+K_1+K_2 \subseteq (IJ:ab)$. We need to prove the reverse inclusion.
 Let $uv \in \mathcal{G}(IJ:ab)$. If $uv \in J$, then we are done. Suppose $uv \notin J$.
 Since $uvab \in IJ$, we have the following cases $ua \in I $ and $vb \in J$ or 
 $ua \in J $ and $vb \in I$ or $ub \in I $ and $va \in J$ or $ub \in J $ and $va \in I$.
 In all cases, one can show that either $uv \in K_1$ or $uv \in K_2$. Therefore,
 $(IJ:ab)=J+ K_1 + K_2$.
\end{proof}

Let $I$ and $J$ be as in Set-up \ref{setup}.
 Then 
for any $ab \in I$, $\widetilde{(IJ:ab)}$ is a quadratic squarefree monomial ideal, by 
Lemma \ref{stu-lemma}. There exists a graph $\PP$ associated to $\widetilde{(IJ:ab)}$. 
Suppose $xy$ is a minimal generator of $(IJ:ab)$.
If $x \neq y$, then set $\{[x,y]\}=\{x,y\}$  and $\{[x,y]\}$ is
an edge of $\PP$.
If $x=y$, then set $\{[x,y]\}=\{x,z_{x}\}$, where $z_x$ is a new vertex of $\PP$, and $\{[x,y]\}$ is
an edge of $\PP$. 
Observe that
$G$ is a subgraph of $\PP$ i.e.,
$V(G) \subseteq V(\PP)$ and $E(G) \subseteq E(\PP)$. 
For example, let $I=(x_4x_5,x_5x_6,x_4x_6)$ and 
$J=(x_1x_2,x_2x_3,x_3x_4,x_4x_5,x_5x_6,x_1x_6,x_4x_6)$. Then $(IJ:x_4x_5)=J+(x_6^2,x_3x_6) \subset 
\K[x_1,\ldots,x_6]$ and $\widetilde{(IJ:x_4x_5)}=J+(x_6z_{x_6},x_3x_6) 
\subset \K[x_1,\ldots,x_6,z_{x_6}]$. Let $\PP$ be the graph associated to $\widetilde{(IJ:x_4x_5)}$.
Then $V(\PP)=V(G) \cup \{z_{x_6}\}$ and $E(\PP)=E(G) \cup \{\{[x_6,x_6]\}, \{x_3,x_6\}\}$.

The following is a useful result on co-chordal graphs that allow us to assume 
  certain order on their edges.

\begin{lemma}\cite[Lemma 1 and Theorem 2]{benzaken}\label{main-lemmatech}
 Let $G$ be a graph and $E(G)=\{e_1\ldots,,e_t\}$. Then $G$ is a co-chordal graph if and only  if
 there is an ordering of edges of $G$,
 $e_{i_1}< \cdots <e_{i_t}$, such that for $1 \leq r \leq t$,
 $(V(G),\{e_{i_1},\ldots,e_{i_r}\})$ has no induced subgraph isomorphic to
 $2K_2$.
\end{lemma}

One of the key ingredients in the proof of the main results is a new graph, $\PP$, obtained
from the  given graphs $G$ and $H$ as in Lemma \ref{stu-lemma}. Our main aim in this section is to get an
upper bound for the co-chordal cover number of $\PP$ 
 which in turn will help
us in bounding $\reg(IJ)$. For this purpose, we need to understand the structure of
the graph $\PP$ in more detail. First we discuss the procedure to get a new graph from the given 
co-chordal subgraph of $G$.
 
 \begin{dis}\label{dis}
Let $I$ and $J$ be as in Set-up \ref{setup}.
Let $\PP$ be the graph 
associated to 
 $\widetilde{(IJ:ab)}$ for any $ab\in I$.
  Suppose $\cochord(G)=\widetilde{n}$. Then there exist co-chordal subgraphs 
  $H_1,\ldots,H_{\widetilde{n}}$ of $G$ 
 such that $E(G)= \bigcup \limits_{i=1}^{\widetilde{n}} E(H_i)$. 
 Let $N_H(a) \setminus b=\{a_1,\ldots,a_{\alpha'}\}$,
 $N_G(a)\setminus b=\{a_1,\ldots, a_{\alpha'},a_{\alpha'+1}, \ldots,a_\alpha\}$,
  $N_H(b) \setminus 
  a=\{b_1,\ldots,b_{\beta'}\}$ and $N_G(b)\setminus a=\{b_1,\ldots,b_{\beta'}, b_{\beta'+1},\ldots,
  b_\beta\}$.
  Set 
  \begin{align*}
    \mathcal{N}(G)_a&=\{ \{a,a_i\} \in E(G) \mid 1 \leq i \leq \alpha\} \text{ and }
    \mathcal{N}(G)_b=\{ \{b,b_i\} \in E(G) \mid 1 \leq i \leq \beta\}.
  \end{align*} 
  Note that if $c \in (N_G(a) \setminus b) \cap (N_G(b) \setminus a)$, then
$\{a,c\} \in \mathcal{N}(G)_a$ and $\{b,c\} \in \mathcal{N}(G)_b$.
Since $H_m$ is co-chordal for all $1 \leq m \leq \widetilde{n}$, by Lemma \ref{main-lemmatech},
there is an ordering of edges of $H_m$, 
\begin{align}\label{ori-ord}
 f_1 < \cdots < f_{t_m},
\end{align}
 such
that for $1\leq r \leq t_m, ~ (V(H_m), \{f_1,\ldots,f_r\})$ has no
induced subgraph isomorphic to $2K_2$.

We now define a procedure to add certain 
edges to $H_m$, to get a new graph $H_m'$ in the following steps:

\begin{enumerate}
 \item[\textsc{Step 1:}] 
 If $f_k=\{a,b\} $ for some $1 \leq k \leq t_m$, 
then we extend the ordered sequence of edges $f_i$s by entering some new edges in the 
following order:
  \begin{align*}
   &\cdots <f_k <\{a,a_1\}< \cdots< \{a,a_{\alpha'}\}<\{b,b_1\}< \cdots< 
   \{b,b_{\beta'}\}<  \{[a_{1},b_1]\} < \cdots 
   < \{[a_1,b_{\beta'}]\}\\ &<
   \{[a_2,b_1]\} < \cdots < \{[a_2,b_{\beta'}]\} < \cdots
   < \{[a_{\alpha'},b_{1}]\}<\cdots <\{[a_{\alpha'},b_{\beta'}]\}<f_{k+1}< \cdots
  \end{align*}
  \item[\textsc{Step 2:}] (i) If for $1 \leq \mu \leq \alpha$, $f_{k_1}=\{a,a_{\mu}\} \in \mathcal{N}(G)_a$ 
  for some $1 \leq k_1 \leq t_m$, 
  then extend the ordered sequence of edges obtained in \textsc{Step 1} by adding some new edges in the following order: 
$$ \cdots <f_{k_1} <\{[a_{\mu},b_1]\} < \cdots < \{[a_{\mu},b_{\beta'}]\}<f_{k_{1}+1}< \cdots$$ 
\vskip 1mm \noindent
(ii) If for $1 \leq \mu \leq \beta$, $f_{k_2}=\{b,b_{\mu}\} \in \mathcal{N}(G)_b$ for some $1 \leq k_2 \leq t_m$, 
then extend the ordered sequence obtained from \textsc{Step 2}(i) by adding new edges in the following order:
  $$\cdots <f_{k_2} <\{[b_{\mu},a_1]\} < \cdots < \{[b_{\mu},a_{\alpha'}]\}<f_{k_2+1}< \cdots$$     
  otherwise do not do anything.
\item[\textsc{Step 3:}] 
 After applying \textsc{Step 1} and \textsc{Step 2}, we get that 
 the ordered sequence
  \begin{align}\label{ord1}
    g_1< \cdots < g_{t_{m'}}
  \end{align}
of whose elements
are edges of  $H_m'$. Note that 
 these steps give us an ordered sequence of edges where some edges may appear more than once
 i.e., $g_i$ may be equal to $g_j$ for some $1 \leq i, j \leq t_{m'}$
 in $\eqref{ord1}$.
For each edge we keep the first appearance and delete the subsequent ones 
 in $\eqref{ord1}$
 to get a  non repeating ordered sequence 
 \begin{align*}\label{f-ord}
   \mathfrak{g}_1< \cdots < \mathfrak{g}_{t_{m_1}}
  \end{align*}
 of edges of $H_m'$ where $t_{m_1} \leq t_{m'}$. 
 \end{enumerate}
 \end{dis}
 First note that $\{g_1,\ldots,g_{t_{m'}}\}=\{\mathfrak{g}_1,\ldots,\mathfrak{g}_{t_{m_1}}\}$. 
For the convenience of the readers, we give an example 
 in below, describing the ordering that are defined  above.

 \begin{example}
  Let $G$ and $H$ be the graphs as shown in the figure below.
Set $I=I(H)$, $J=I(G)$, $a=x_7$ and $b=x_6$. 
 Let $\PP$ be the graph associated to $\widetilde{(IJ:ab)}$. 
 Note that $N_G(x_6)\setminus \{x_7\}=\{x_5,x_8,x_{10}\}$,
 $N_G(x_7)\setminus \{x_6\}=\{x_2, x_4,x_8\}$,
 $N_H(x_6)\setminus \{x_7\}=\{x_5,x_8\}$ and 
 $N_H(x_7)\setminus \{x_6\}=\{x_4,x_8\}$. 
Let $H_1$, $H_2$ and $H_3$ be  co-chordal
 subgraphs  of $G$ such that 
 $E(G)=\bigcup_{i=1}^3 E(H_i)$. Therefore $\cochord(G)=3$.

\begin{figure}[H]
\begin{tikzpicture}[scale=0.25]
\draw (4,12)-- (8,12);
\draw (2,10)-- (4,12);
\draw (1.9,7.69)-- (2,10);
\draw (8,12)-- (10,10);
\draw (10,10)-- (10,8);
\draw (2,6)-- (1.9,7.69);
\draw (10,6)-- (10,8);
\draw (4,4)-- (8,4);
\draw (8,4)-- (10,6);
\draw (2,6)-- (4,4);
\draw (4,4)-- (8,12);
\draw (4,4)-- (10,8);
\draw (2,10)-- (8,4);
\draw (2,6)-- (8,4);
\draw (16,4)-- (20,4);
\draw (20,4)-- (22,6);
\draw (22,6)-- (22,8);
\draw (16,4)-- (14,6);
\draw (14,6)-- (14,8);
\draw (14,6)-- (20,4);
\draw (16,4)-- (22,8);
\draw (30,4)-- (34,4);
\draw (30,4)-- (28,6);
\draw (28,8)-- (28,6);
\draw (28,10)-- (28,8);
\draw (28,10)-- (30,12);
\draw (30,12)-- (34,12);
\draw (34,12)-- (36,10);
\draw (36,8)-- (36,10);
\draw (36,6)-- (36,8);
\draw (36,6)-- (34,4);
\draw (30,4)-- (34,12);
\draw (34,4)-- (28,6);
\draw (28,10)-- (34,4);
\draw (30,4)-- (36,8);
\draw [line width=1.2pt,dash pattern=on 2pt off 2pt] (28,6)-- (34,12);
\draw [line width=1.2pt,dash pattern=on 2pt off 2pt] (34,12)-- (36,6);
\draw [shift={(28,8)},line width=1.2pt,dash pattern=on 2pt off 2pt]  plot[domain=1.57:4.71,variable=\t]({1*2*cos(\t r)+0*2*sin(\t r)},{0*2*cos(\t r)+1*2*sin(\t r)});
\draw [line width=1.2pt,dash pattern=on 2pt off 2pt] (28,10)-- (36,8);
\draw [line width=1.2pt,dash pattern=on 2pt off 2pt] (26,4)-- (28,6);
\draw [line width=1.2pt,dash pattern=on 2pt off 2pt] (28,6)-- (36,8);
\draw [line width=1.2pt,dash pattern=on 2pt off 2pt] (36,6)-- (28,6);
\draw (5.36,3.4) node[anchor=north west] {$G$};
\draw (17.71,3.58) node[anchor=north west] {$H$};
\draw (31.81,3.14) node[anchor=north west] {$\mathcal{P}$};
\begin{scriptsize}
\fill [color=black] (4,12) circle (3.5pt);
\draw[color=black] (4.36,12.56) node {$x_1$};
\fill [color=black] (8,12) circle (3.5pt);
\draw[color=black] (8.34,12.56) node {$x_2$};
\fill [color=black] (2,10) circle (3.5pt);
\draw[color=black] (1.2,9.97) node {$x_{10}$};
\fill [color=black] (1.9,7.69) circle (3.5pt);
\draw[color=black] (1.2,7.87) node {$x_9$};
\fill [color=black] (10,10) circle (3.5pt);
\draw[color=black] (10.36,10.58) node {$x_3$};
\fill [color=black] (10,8) circle (3.5pt);
\draw[color=black] (10.71,8.57) node {$x_4$};
\fill [color=black] (2,6) circle (3.5pt);
\draw[color=black] (1.2,6.21) node {$x_8$};
\fill [color=black] (10,6) circle (3.5pt);
\draw[color=black] (10.76,6.56) node {$x_5$};
\fill [color=black] (4,4) circle (3.5pt);
\draw[color=black] (3.83,3.53) node {$x_7$};
\fill [color=black] (8,4) circle (3.5pt);
\draw[color=black] (8.03,3.45) node {$x_6$};
\fill [color=black] (16,4) circle (3.5pt);
\draw[color=black] (15.79,3.45) node {$x_7$};
\fill [color=black] (20,4) circle (3.5pt);
\draw[color=black] (20.08,3.53) node {$x_6$};
\fill [color=black] (22,6) circle (3.5pt);
\draw[color=black] (22.7,6.56) node {$x_5$};
\fill [color=black] (22,8) circle (3.5pt);
\draw[color=black] (22.35,8.57) node {$x_4$};
\fill [color=black] (14,6) circle (3.5pt);
\draw[color=black] (13.33,6.07) node {$x_8$};
\fill [color=black] (14,8) circle (3.5pt);
\draw[color=black] (14.34,8.57) node {$x_9$};
\fill [color=black] (30,4) circle (3.5pt);
\draw[color=black] (30.06,3.49) node {$x_7$};
\fill [color=black] (34,4) circle (3.5pt);
\draw[color=black] (34.13,3.42) node {$x_6$};
\fill [color=black] (28,6) circle (3.5pt);
\draw[color=black] (27.96,5.26) node {$x_8$};
\fill [color=black] (28,8) circle (3.5pt);
\draw[color=black] (27.37,8.09) node {$x_9$};
\fill [color=black] (28,10) circle (3.5pt);
\draw[color=black] (27.61,10.58) node {$x_{10}$};
\fill [color=black] (30,12) circle (3.5pt);
\draw[color=black] (30.37,12.56) node {$x_1$};
\fill [color=black] (34,12) circle (3.5pt);
\draw[color=black] (34.4,12.56) node {$x_2$};
\fill [color=black] (36,10) circle (3.5pt);
\draw[color=black] (36.37,10.58) node {$x_3$};
\fill [color=black] (36,8) circle (3.5pt);
\draw[color=black] (36.67,8.57) node {$x_4$};
\fill [color=black] (36,6) circle (3.5pt);
\draw[color=black] (36.67,6.56) node {$x_5$};
\fill [color=black] (26,4) circle (1.5pt);
\draw[color=black] (27.35,3.88) node {$z_{x_8}$};
\end{scriptsize}
\end{tikzpicture}
\end{figure}

Let $f_1=\{x_1,x_2\}<f_2=\{x_2,x_7\}<f_3=\{x_2,x_3\}<f_4=\{x_3,x_4\}$
be the ordering of the edges of
 $H_1$ such that for $1 \leq i \leq 4$, $(V(H_1), \{f_1,\ldots,f_i\})$ has
  no induced subgraph isomorphic to $2K_2$.
Note that $f_{i}\neq \{a,b\}$ for all $1 \leq i \leq 4$. Therefore there is no change in the ordered sequence
of edges $f_i$'s. Since $f_2 \in \mathcal{N}(G)_a$, by \textsc{Step 2}(i),
\[
 f_1<f_2<\{x_2,x_5\}<\{x_2,x_8\}<f_3<f_4.
\]
Also note that $f_i \notin \mathcal{N}(G)_b$ for all $1 \leq i \leq 4$.
Since no repeated edge in the above ordering, by \textsc{Step 3},
 $H_1'$ is the graph with edge set 
$E(H_1) \cup \{\{x_2,x_5\}, \{x_2,x_8\}\}$ and whose edges appearing in the above ordered sequence.

\begin{figure}[H]
\begin{tikzpicture}[scale=.3]
\draw (4,12)-- (6,12);
\draw (6,10)-- (6,12);
\draw (6,8)-- (6,10);
\draw (4,10)-- (6,12);
\draw (10,12)-- (8,10);
\draw (8,10)-- (8,8);
\draw (10,8)-- (8,8);
\draw (10,10)-- (8,10);
\draw (14,8)-- (16,8);
\draw (12,10)-- (14,8);
\draw (18,10)-- (16,8);
\draw (18,10)-- (18,12);
\draw (18,12)-- (14,8);
\draw (12,10)-- (16,8);
\draw (20,12)-- (22,12);
\draw (20,10)-- (22,12);
\draw (22,12)-- (22,10);
\draw (22,8)-- (22,10);
\draw [line width=1.2pt,dash pattern=on 2pt off 2pt] (22,12)-- (22,14);
\draw [line width=1.2pt,dash pattern=on 2pt off 2pt] (22,12)-- (24,12);
\draw (26,10)-- (28,12);
\draw (26,10)-- (28,10);
\draw (26,10)-- (26,8);
\draw (28,8)-- (26,8);
\draw [line width=1.2pt,dash pattern=on 2pt off 2pt] (26,10)-- (28,8);
\draw [line width=1.2pt,dash pattern=on 2pt off 2pt] (26,12)-- (26,10);
\draw (32,8)-- (34,8);
\draw (30,10)-- (32,8);
\draw (34,8)-- (36,10);
\draw (36,12)-- (36,10);
\draw (32,8)-- (36,12);
\draw (30,10)-- (34,8);
\draw [line width=1.2pt,dash pattern=on 2pt off 2pt] (30,10)-- (36,10);
\draw [line width=1.2pt,dash pattern=on 2pt off 2pt] (30,10)-- (36,12);
\draw [line width=1.2pt,dash pattern=on 2pt off 2pt] (30,12)-- (30,10);
\draw (4.26,7.56) node[anchor=north west] {$H_1$};
\draw (8.2,7.3) node[anchor=north west] {$H_2$};
\draw (14.09,7.12) node[anchor=north west] {$H_3$};
\draw (20.6,7.68) node[anchor=north west] {$H_1'$};
\draw (26.52,7.68) node[anchor=north west] {$H_2'$};
\draw (32.37,7.56) node[anchor=north west] {$H_3'$};
\begin{scriptsize}
\fill [color=black] (4,12) circle (3.5pt);
\draw[color=black] (4.3,12.51) node {$x_1$};
\fill [color=black] (6,12) circle (3.5pt);
\draw[color=black] (6.29,12.51) node {$x_2$};
\fill [color=black] (6,10) circle (3.5pt);
\draw[color=black] (5.39,10.12) node {$x_3$};
\fill [color=black] (6,8) circle (3.5pt);
\draw[color=black] (5.39,8.5) node {$x_4$};
\fill [color=black] (4,10) circle (3.5pt);
\draw[color=black] (3.49,10.21) node {$x_7$};
\fill [color=black] (10,12) circle (3.5pt);
\draw[color=black] (10.27,12.51) node {$x_{1}$};
\fill [color=black] (8,10) circle (3.5pt);
\draw[color=black] (7.66,10.52) node {$x_{10}$};
\fill [color=black] (8,8) circle (3.5pt);
\draw[color=black] (7.93,7.56) node {$x_9$};
\fill [color=black] (10,8) circle (3.5pt);
\draw[color=black] (9.99,7.56) node {$x_8$};
\fill [color=black] (10,10) circle (3.5pt);
\draw[color=black] (10.19,10.52) node {$x_6$};
\fill [color=black] (14,8) circle (3.5pt);
\draw[color=black] (13.89,7.52) node {$x_7$};
\fill [color=black] (16,8) circle (3.5pt);
\draw[color=black] (16.23,7.48) node {$x_6$};
\fill [color=black] (12,10) circle (3.5pt);
\draw[color=black] (12.33,10.52) node {$x_8$};
\fill [color=black] (18,10) circle (3.5pt);
\draw[color=black] (18.8,10.22) node {$x_5$};
\fill [color=black] (18,12) circle (3.5pt);
\draw[color=black] (18.3,12.51) node {$x_4$};
\fill [color=black] (20,12) circle (3.5pt);
\draw[color=black] (20.32,12.51) node {$x_1$};
\fill [color=black] (22,12) circle (3.5pt);
\draw[color=black] (22.71,12.51) node {$x_2$};
\fill [color=black] (20,10) circle (3.5pt);
\draw[color=black] (19.93,10.64) node {$x_7$};
\fill [color=black] (22,10) circle (3.5pt);
\draw[color=black] (22.67,10.52) node {$x_3$};
\fill [color=black] (22,8) circle (3.5pt);
\draw[color=black] (22.67,8.5) node {$x_4$};
\fill [color=black] (22,14) circle (3.5pt);
\draw[color=black] (22.31,14.5) node {$x_8$};
\fill [color=black] (24,12) circle (3.5pt);
\draw[color=black] (24.3,12.51) node {$x_5$};
\fill [color=black] (26,10) circle (3.5pt);
\draw[color=black] (25.09,10.09) node {$x_{10}$};
\fill [color=black] (28,12) circle (3.5pt);
\draw[color=black] (28.31,12.51) node {$x_1$};
\fill [color=black] (28,10) circle (3.5pt);
\draw[color=black] (28.59,10.52) node {$x_6$};
\fill [color=black] (26,8) circle (3.5pt);
\draw[color=black] (25.25,8.07) node {$x_9$};
\fill [color=black] (28,8) circle (3.5pt);
\draw[color=black] (28.59,8.5) node {$x_8$};
\fill [color=black] (26,12) circle (3.5pt);
\draw[color=black] (26.6,12.51) node {$x_4$};
\fill [color=black] (32,8) circle (3.5pt);
\draw[color=black] (31.82,7.68) node {$x_7$};
\fill [color=black] (34,8) circle (3.5pt);
\draw[color=black] (34.82,7.83) node {$x_6$};
\fill [color=black] (30,10) circle (3.5pt);
\draw[color=black] (30.08,9.19) node {$x_8$};
\fill [color=black] (36,10) circle (3.5pt);
\draw[color=black] (36.58,10.52) node {$x_5$};
\fill [color=black] (36,12) circle (3.5pt);
\draw[color=black] (36.46,12.51) node {$x_4$};
\fill [color=black] (30,12) circle (3.5pt);
\draw[color=black] (30.46,12.51) node {$z_{x_8}$};
\end{scriptsize}
\end{tikzpicture}
\end{figure}

 Let  $f_1'=\{x_1,x_{10}\}<f'_2=\{x_6,x_{10}\}<f'_3=\{x_{9},x_{10}\}<f'_4=\{x_9,x_8\}$
be the ordering of the edges of
 $H_2$ such that for $1 \leq i \leq 4$, $(V(H_2), \{f_1',\ldots,f_i'\})$ has
  no induced subgraph isomorphic to $2K_2$. Note that
  $f_i' \neq \{a,b\}$ and $f_i' \notin \mathcal{N}(G)_a$ for all $1 \leq i \leq 4$.
   Since $f_2' \in \mathcal{N}(G)_{x_{6}}$, 
 by \textsc{Step 2}(ii) we have
 \[
  f_1'<f_2'<\{x_{10},x_4\}<\{x_{10},x_8\}<f_3'<f_4'.
 \]
In this case also no repeated edges. By \textsc{Step 3},
 $H_2'$ is the graph with edge set 
$E(H_2) \cup \{\{x_{10},x_4\}, \{x_{10},x_8\}\}$ and whose edges appearing in the above ordered sequence.

Let 
\begin{align*}
 f''_1=\{x_7,x_{6}\}<f''_2=\{x_6,x_{5}\}<f''_3=\{x_{5},x_{4}\}&<f''_4=\{x_4,x_7\}<\\
 &f''_5=\{x_7,x_8\}<f''_6=\{x_6,x_8\}
\end{align*}
be the ordering of the edges of
 $H_3$ such that for $1 \leq i \leq 6$, $(V(H_3), \{f_1'',\ldots,f_i''\})$ has
  no induced subgraph isomorphic to $2K_2$. Since $f_1'' =\{a,b\}$, by 
  \textsc{Step 1},
  \begin{align*}
   &f_1''=\{x_7,x_6\}<\{x_7,x_8\}<\{x_7,x_4\}<\{x_6,x_5\}
     <\{x_6,x_8\}<\{[x_8,x_8]\}<
    \\&\{x_8,x_5\}<\{x_4,x_5\} <\{x_4,x_8\}  
    <f_2''<f_3''<f_4''
    <f_5''<f_6''.
  \end{align*}
Since $f_4'', f_5'' \in \mathcal{N}(G)_a$, by \textsc{Step 2}(i), we have
\begin{align*}
   &f_1''=\{x_7,x_6\}<\{x_7,x_8\}<\{x_7,x_4\}<\{x_6,x_5\}
     <\{x_6,x_8\}<\{[x_8,x_8]\}<
    \\&\{x_8,x_5\}<\{x_4,x_5\} <\{x_4,x_8\}  
    <f_2''<f_3''<f_4''<\{x_4,x_5\}<\{x_4,x_8\}
    <f_5''\\&<\{[x_8,x_8]\}<\{x_8,x_6\}<f_6''.
\end{align*}   
  Since $f_2'', f_6'' \in \mathcal{N}(G)_b$, by \textsc{Step 2}(ii), we have
  \begin{align*}
   &f_1''=\{x_7,x_6\}<\{x_7,x_8\}<\{x_7,x_4\}<\{x_6,x_5\}
     <\{x_6,x_8\}<\{[x_8,x_8]\}< 
    \{x_8,x_5\}\\&<\{x_4,x_5\} <\{x_4,x_8\}  
    <f_2''=\{x_6,x_5\}<\{x_5,x_4\}<\{x_5,x_8\}<f_3''=\{x_5,x_4\}\\&<f_4''=\{x_4,x_7\}<\{x_4,x_5\}<\{x_4,x_8\}
    <f_5''=\{x_7,x_8\}<\{[x_8,x_8]\}<\{x_8,x_6\}\\&<f_6''=\{x_6,x_8\}<\{[x_8,x_8]\}<\{x_8,x_4\}.
\end{align*} 
Since the edges $\{x_6,x_5\}$, $\{x_5,x_4\}$, $\{x_5,x_8\}$, $\{x_4,x_7\}$,
$\{x_4,x_5\}$, $\{x_4,x_8\}$, $\{x_7,x_8\}$, $\{x_6,x_8\}$,
  $\{[x_8,x_8]\}$ are repeated in the above ordering,
  by \textsc{Step 3}, we have
  \begin{align*}
   &\{x_7,x_6\}<\{x_7,x_8\}<\{x_7,x_4\}<\{x_6,x_5\}
     <\{x_6,x_8\}<\{[x_8,x_8]\}< 
    \{x_8,x_5\}\\&<\{x_4,x_5\} <\{x_4,x_8\}  
\end{align*}
Therefore,   $H_3'$ is the graph with edge set $E(H_3')=E(H_3) \cup \{ \{x_8,x_4\},\{x_8,x_5\},\{x_8,z_{x_8}\}\}$
 and whose edges appearing in the above ordered sequence.

 \end{example}

The operations used in \textsc{Step 1} and \textsc{Step 2} above will also be used subsequently.
 So we fix notation to refer to them. Subsequently
we shall use these operations repeatedly. Instead of separately describe them on each occasion we shall
simply refer to the operation number.
 \begin{enumerate}
 
  \item[Op 1:]
  The operation used in \textsc{Step 2}(i) i.e.,
  if for $1 \leq \mu \leq \alpha$, $f_{k_1}=\{a,a_{\mu}\} \in \mathcal{N}(G)_a$ 
  for some $1 \leq k_1 \leq t_m$, 
  then  
$$ \cdots <f_{k_1} <\{[a_{\mu},b_1]\} < \cdots < \{[a_{\mu},b_{\beta'}]\}<f_{k_1+1}< \cdots$$
  
\item[Op 2:] The operation used in \textsc{Step 2}(ii) i.e., if for $1 \leq \mu \leq \beta$, $f_{k_2}=\{b,b_{\mu}\}
\in \mathcal{N}(G)_b$ for some $1 \leq k_2 \leq t_m$, 
then 
  $$\cdots <f_{k_2} <\{[b_{\mu},a_1]\} < \cdots < \{[b_{\mu},a_{\alpha'}]\}<f_{k_2+1}< \cdots$$     
  
  \item[Op 3:] The operation used in \textsc{Step 1} i.e., if $f_k=\{a,b\}$ for some
  $1 \leq k \leq t_m$, then  
  \begin{align*}
   & \cdots <f_k <\{a,a_1\}< \cdots< \{a,a_{\alpha'}\}<\{b,b_1\}< \cdots< \{b,b_{\beta'}\}<  \{[a_{1},b_1]\} < \cdots < \{[a_1,b_{\beta'}]\}< 
   \\ &\{[a_2,b_1]\} < \cdots < \{[a_2,b_{\beta'}]\} < \cdots 
  < \{[a_{\alpha'},b_{1}]\}<\cdots <\{[a_{\alpha'},b_{\beta'}]\}<f_{k+1}< \cdots
  \end{align*}
\end{enumerate}
 We call the added edges in the above operations
  as \textit{new edges}.

 We make some observations which follows directly from the above discussion.

\begin{obs}\label{oequal} We use the same notation as in Discussion \ref{dis}.
 \begin{enumerate}
 \item 
 Let $\PP$ be the graph associated to 
 $\widetilde{(IJ:ab)}$ for any $ab \in I$. 
 Let $h=\{[c,d]\}$ be  a new edge as in 
 Op 1, Op 2 or Op 3. First note that
 $h \in \mathcal{N}(G)_a \cup \mathcal{N}(G)_{b}$
 or  $c \in N_G(a)$, $d \in N_H(b)$ or 
 $c \in N_H(a)$, $d \in N_G(b)$.
 It follows from 
 Lemma \ref{stu-lemma} that $h \in E(\PP)$. Therefore
 $H_m'$ is a subgraph of $\PP$ for all $1 \leq m \leq \widetilde{n}$.
 Hence $\bigcup\limits_{1 \leq m \leq \widetilde{n}} E(H_m') \subseteq E(\PP)$.
 It is also not hard to verify that $E(\PP) \subseteq \bigcup\limits_{1 \leq m 
 \leq \widetilde{n}} E(H_m')$.
 Therefore $E(\PP)=\bigcup\limits_{1 \leq m \leq \widetilde{n}} E(H_m')$.

 \item Let $g_1<\cdots<g_{t_{m'}}$ be the  ordered sequence whose elements are edges of $H_m'$
 as in \eqref{ord1}.
 Suppose $g_i$ is a new edge as in Op 1 (Op 2 or Op 3)  where $1 \leq i \leq t_{m'}$. Then there exists $g_{i'} \in \mathcal{N}(G)_a$ ($g_{i'} \in \mathcal{N}(G)_b$ or 
 $g_{i'}=\{a,b\}$)
 such that 
 $g_{i'}<g_i$ i.e., $g_1< \cdots <g_{i'}<\cdots <g_i< \cdots< g_{t_{m_1}}$.
 
 \end{enumerate}
\end{obs}

 Now we fix some notation for some of the technical lemmas that are needed for the proof of
the main result.

\begin{notation}
 We use the same notation as in Discussion \ref{dis}. Let 
 $g_1< \cdots < g_{t_{m'}}$ be the ordered sequence 
 whose elements are edges of $H_m'$ as in \eqref{ord1}. 
 For $1 \leq i \leq t_{m'}$, let
  $\mathcal{K}_i$ denote the graph with edge set $\{g_1,\ldots,g_i\}$ 
 and  whose edges appearing in the following ordered sequence
 $g_1<\cdots<g_i$.
\end{notation}

In the next two lemmas, we further  reveals the structure of $H_m'$.
 \begin{lemma}\label{old-new}
 We use the same notation as in Discussion \ref{dis}.
 If $\mathcal{K}_i$ has no induced subgraph isomorphic to $2K_2$ for all 
 $1 \leq i \leq t_{m'}$, then $(V(H_m'),\{\mathfrak{g}_1,\ldots,\mathfrak{g}_j\})$
  has no induced subgraph isomorphic to $2K_2$ for all $1 \leq j \leq t_{m_1}$.
 \end{lemma}
\begin{proof}
 Suppose $(V(H_m'),\{\mathfrak{g}_1,\ldots,\mathfrak{g}_q\})$ has an induced subgraph
 isomorphic to $2K_2$, say $\{\mathfrak{g}_p, \mathfrak{g}_q\}$, for some $1 \leq p< q \leq t_{m_1}$.
 Set $\mathfrak{g}_q=g_s$ for some $1 \leq s \leq t_{m'}$. It can also noted that
 $\mathfrak{g}_p=g_r<g_s$ for some $1 \leq r < s$. 
 Since $\{g_r,g_s\}$ can not form an induced subgraph $2K_2$ in $\mathcal{K}_i$ for all 
 $1 \leq i \leq t_{m'}$, 
 $g_r$ and $g_s$ have a vertex in common or there exist an edge 
$f_l \in E(H_m)$ such
that $f_l < g_{s}$ connecting $g_{r}$ and $g_{s}$.
Note that $f_l \in \{\mathfrak{g}_1,\ldots,\mathfrak{g}_q\})$.
 Both cases we get a contradiction to the assumption. Therefore
 $(V(H_m'),\{\mathfrak{g}_1,\ldots,\mathfrak{g}_j\})$
  has no induced subgraph isomorphic to $2K_2$ for all $1 \leq j \leq t_{m_1}$.
\end{proof}
 
 \begin{lemma}\label{both-edge}
  We use the same notation as in Discussion \ref{dis}. 
  If $\mathcal{K}_j$ has an induced subgraph isomorphic to $2K_2$, say $\{g_i,g_j\}$, for some 
 $1 \leq i< j \leq t_{m'}$, then $g_i,g_j \notin E(H_m)$.
 \end{lemma}
\begin{proof}
 Let $f_1<\cdots< f_{t_m}$ be 
 the ordering of edges of
  $H_m$ as in \eqref{ori-ord}.
 Suppose $g_i,g_j \in E(H_m)$. Set $f_p=g_i$ and $f_q=g_j$ for some $1 \leq p< q \leq t_m$.
 Note that $(V(H_m),\{f_1,\ldots,f_r\})$ has no induced subgraph isomorphic to $2K_2$
 for all $1 \leq r \leq t_m$. Since $g_{i}, g_{j} \in E(H_m)$, by Lemma \ref{main-lemmatech}, they can not form an induced $2K_2$-subgraph of $H_m$. 
Therefore, either $g_{i}$ and $g_{j}$ have a vertex in common or there exist an edge 
$f_l \in E(H_m)$ such
that $f_l < g_{j}$ connecting $g_{i}$ and $g_{j}$. If 
$g_i$ and $g_j$ have a vertex in common in $H_m$, then  this
contradicts the assumption that $\{g_i , g_j\}$ forms an induced $2K_2$-subgraph in 
$\mathcal{K}_j$. If $f_l$ is an edge connecting $g_i$ and $g_j$, then $f_l \in E(\mathcal{K}_j)$.
This is a contradiction to $g_i,g_j \in E(H_m)$. Therefore $g_i,g_j \notin E(H_m)$.
 \end{proof}
 
 Now we prove that the co-chordal cover number of $\PP$ is bounded above by that of $G$.
\begin{lemma}\label{cochord-lemma}
Let $I$ and $J$ be as in Set-up \ref{setup}.
Let $\PP$ be the graph 
associated to 
 $\widetilde{(IJ:ab)}$ for any $ab \in I$. Then
 \[
  \cochord(\PP) \leq \cochord(G).
 \]
\end{lemma}
\begin{proof}
Let $\cochord(G)=\widetilde{n}$. 
Then there exist co-chordal subgraphs $H_1,\ldots,H_{\widetilde{n}}$ of $G$ 
 such that $E(G)= \bigcup \limits_{i=1}^{\widetilde{n}} E(H_i)$. If $E(G)=E(\PP)$, then 
 we are done.  Suppose $E(G) \neq E(\PP)$. We use the same notation as in Discussion \ref{dis}.
 Since $H_m$ is co-chordal, by Lemma \ref{main-lemmatech}, there is an 
ordering of the edges of $H_m$,
 $f_1 < \cdots < f_{t_m},$
 such that for $1\leq r \leq t_{m}, ~ (V(H_m), \{f_1,\ldots,f_r\})$ has no
induced subgraph isomorphic to $2K_2$. 
By Observation \ref{oequal} (1),
we have $E(\PP) = \bigcup\limits_{m=1}^{\widetilde{n}} E(H_m')$.
 Let  $g_1< \cdots < g_{t_{m'}}$
 be the ordered sequence of  edges of
 $H_m'$
  as in $\eqref{ord1}$.
  Now we claim that $\mathcal{K}_r$ has no induced
  subgraph isomorphic to $2K_2$ for all $1 \leq r \leq t_{m'}$. Suppose not i.e.,
  there exists a least $j$ such that
$\mathcal{K}_j$ has an induced $2K_2$-subgraph, say
$\{g_i, g_j\}$ for some $i<j$. First note that  both $g_i$ and $g_j$ can not be  new edges as in Op 1, Op 2, Op 3.
By Lemma \ref{both-edge}, $g_i,g_j \notin E(H_m)$. Therefore, we have the following cases:
\begin{enumerate}
 \item $g_i \in E(H_m)$, $g_j$ is a new edge as in Op 1 or 
 $g_i$ is a new edge as in Op 1, $g_j \in E(H_m)$;  
 \item $g_i \in E(H_m)$, $g_j$ is a new edge as in Op 2 or 
 $g_i$ is a new edge as in Op 2, $g_j \in E(H_m)$;
 \item $g_i \in E(H_m)$, $g_j$ is a new edge as in Op 3 or 
 $g_i$ is a new edge as in Op 3, $g_j \in E(H_m)$;
 \item $g_i$ is a new edge as in Op 1, $g_j$ is a new edge as in Op 2 or 
 $g_i$ is a new edge as in Op 2, 
 $g_j$ is a new edge as in Op 1;
 \item $g_i$ is a new edge as in Op 1, $g_j$ is a new edge as in Op 3 or 
 $g_i$ is a new edge as in Op 3, 
 $g_j$ is a new edge as in Op 1;
 \item $g_i$ is a new edge as in Op 2, $g_j$ is a new edge as in Op 3 or 
 $g_i$ is a new edge as in Op 3, 
 $g_j$ is a new edge as in Op 1;
\end{enumerate}
\vskip 1mm \noindent
\textsc{Case 1:} Suppose $g_i \in E(H_m)$ and 
$g_j$ is a new edge as in Op 1.
Let $g_i=\{u,v\} \in E(H_m)$ and 
$g_j=\{[a_\mu, b_p]\}$
  for some $1 \leq \mu \leq \alpha$, $1 \leq p \leq \beta' $.
  By Op 1, we have
  $g_{j'}=\{a, a_{\mu}\}<g_j.$
 Since $g_i,g_{j'} \in E(H_m)$, they can not form an induced $2K_2$-subgraph of $H_m$. 
Therefore, either $g_{j'}$ and $g_i$ have a vertex in common or there exist an edge $g_l \in E(H_m)$ such
that $g_l < g_{j'}$ connecting $g_i$ and $g_{j'}$. If $g_i$ and $g_{j'}$ have a vertex in common, then this contradicts
the assumption that $\{g_i,g_j\}$ forms an induced $2K_2$-subgraph. Suppose $g_l$ is an edge connecting 
$g_i$ and $g_{j'}$. Let $g_l=\{u,a\}$ and $u \neq b$. Then $g_l \in \mathcal{N}(G)_a$. By Op 1, 
$g_l<\{[u,b_p]\}$. 
We have $g_l<\{[u,b_p]\}<g_{j'}<g_j$.
This is a contradiction to $\{g_i,g_j\}$ is an induced $2K_2$-subgraph.
If $g_l=\{a,b\}$, then by Op 3, 
 $g_l<\{b,b_p\}$. This also contradicts the assumption that $\{g_i,g_j\}$ is an induced $2K_2$-subgraph.
 Similarly, if $g_l=\{u,a_{\mu}\}$ or $g_l=\{v,a\}$ or $g_l=\{v,a_{\mu}\}$,
 then one arrives at a contradiction. Therefore $\{g_i,g_j\}$ can not form an induced 
  $2K_2$-subgraph of $H_m'$ 

  If $g_i$ is a new edge as in Op 1 and $g_j \in E(H_m)$,
 then we get a contradiction in a similar manner.

 \vskip 1mm \noindent
\textsc{Case 2:} Suppose either $g_i\in E(H_m)$ and $g_j$ is a new edge as in Op 2
or $g_j\in E(H_m)$ and $g_i$
is a new edge as in Op 2.
Proceeding as in \textsc{Case 1}, one can show that 
$g_i$ and $g_j$ can not form an induced $2K_2$-subgraph.
 \vskip 1mm \noindent
\textsc{Case 3:} Suppose $g_i\in E(H_m)$ and $g_j$ is a
new edge as in Op 3.
Let $g_i=\{u,v\} \in E(H_m)$. Then $g_j=\{a,a_\mu\}$ for some $1 \leq \mu \leq \alpha'$
or $g_j=\{b, b_{\mu}\}$ for some $1 \leq \mu \leq \beta'$
or $g_j=\{[a_p, b_q]\}$  for some $1 \leq p \leq \alpha'$, $1 \leq q \leq \beta' $.
 If $g_j=\{a,a_\mu\}$ for some $1 \leq \mu \leq \alpha'$, then by Op 3, 
 we have $g_{j'}=\{a,b\}<g_{j}.$ 
Since $g_i, g_{j'} \in E(H_m)$, they can not form an induced $2K_2$-subgraph of $H_m$. 
Therefore, either $g_{j'}$ and $g_i$ have a vertex in common or there exist an edge $g_l \in E(H_m)$ such
that $g_l < g_{j'}$ connecting $g_i$ and $g_{j'}$.
If $g_i$ and $g_{j'}$ have a vertex in common, then this contradicts
the assumption that $\{g_i,g_j\}$ forms an induced $2K_2$-subgraph.
Suppose $g_l$ is an edge connecting 
$g_i$ and $g_{j'}$. If $g_l=\{b,u\} \in \mathcal{N}(G)_b$, then by Op 2, $
g_l<\{[u,a_{\mu}]\}.$
This also contradicts the assumption that $\{g_i,g_j\}$ is an induced $2K_2$-subgraph.
Similarly, if $g_l=\{v,b\}$ or $g_l=\{u,a\}$ or $g_l=\{u,a\}$, then one arrives at a contradiction. 
If $g_j=\{b, b_{\mu}\}$ for some $1 \leq \mu \leq \beta'$, then we get a contradiction in a similar manner. 

Suppose $g_j=\{[a_p, b_q]\}$  for some $1 \leq p \leq \alpha'$, $1 \leq q \leq \beta' $.
By Op 3, we have
 $g_{j'}=\{a,b\}<g_j.$
Since $g_i, g_{j'} \in E(H_m)$, they can not form an induced $2K_2$-subgraph of $H_m$. 
Therefore, either $g_{j'}$ and $g_i$ have a vertex in common or there exist an edge 
$g_l \in E(H_m)$ such
that $g_l < g_{j'}$ connecting $g_i$ and $g_{j'}$.
Suppose  $g_i$ and $g_{j'}$ have a vertex in common. If $u=a$, then 
$g_i \in \mathcal{N}(G)_a$.
By Op 1,
$g_i<\{[v,b_q]\}$. 
Therefore, we have $g_i<\{[v,b_q]\}<g_{j'}<g_j$.
This is a contradiction to $\{g_i,g_j\}$ forms an induced
$2K_2$-subgraph. Similarly, if $u=b$ or $v=a$ or $v=b$, then one arrives at a contradiction.
Suppose $g_l$ is an edge connecting 
$g_i$ and $g_{j'}$.
Note that $g_l<g_{j'}$.
If $g_l=\{u,a\}$, then 
$g_l \in \mathcal{N}(G)_a$.
By Op 1, $g_l<\{[u, b_q]\}$.
This also contradicts the assumption that $\{g_i,g_j\}$ is an induced $2K_2$-subgraph.
Similarly, if $g_l=\{v,b\}$ or $g_l=\{v,a\}$ or $g_l=\{u,b\}$, then one arrives at a contradiction.

If $g_i$ is a new edge as in Op 3 and $g_j \in E(H_m)$, then we get a contradiction in a similar manner.

\vskip 2mm \noindent
\textsc{Case 4:} Suppose $g_i$ is a new edge as in Op 1 and 
$g_j$ is a new edge as in Op 2. 
Let $g_i=\{[a_p,b_q]\}$  and 
$g_j=\{[a_{p'},b_{q'}]\} $ for some
$1 \leq p \leq \alpha$, $1 \leq q \leq \beta'$, $1 \leq p' \leq \alpha'$, 
$1 \leq q' \leq \beta$. Then by Op 1 and Op 2, 
\[
 g_{i'}=\{a, a_p\}<g_i < g_{j'}=\{b,b_{q'}\}<g_j.
\]
Since $g_{i'}, g_{j'} \in E(H_m)$, they can not form an induced $2K_2$-subgraph of $H_m$. 
Therefore, either $g_{i'}$ and $g_{j'}$ have a vertex in common or there exist an edge $g_l \in E(H_m)$ such
that $g_l < g_{j'}$ connecting $g_{i'}$ and $g_{j'}$.
If  $g_{i'}$ and $g_{j'}$ have a vertex in common, then this contradicts
the assumption that $\{g_i,g_j\}$ forms an induced $2K_2$-subgraph. Suppose $g_l$ is an edge connecting 
$g_{i'}$ and $g_{j'}$. If $g_l=\{a_p,b_{q'}\}$, then this contradicts
the assumption that $\{g_i,g_j\}$ forms an induced $2K_2$-subgraph. 
If $g_l=\{a_p,b\}$, then $g_l \in \mathcal{N}(G)_b$. By Op 2, 
$g_l<\{[a_p,a_{p'}]\}$. This also contradicts the assumption that $\{g_i,g_j\}$ is an induced $2K_2$-subgraph.
Similarly, if $g_l=\{a, b_{q'}\}$, then one arrives at a contradiction. If $g_l=\{a,b\}$, then by Op 3, 
$g_l<\{[a_{p'},b_{q'}]\}$. This also contradicts the assumption that $\{g_i,g_j\}$ is an induced
$2K_2$-subgraph. 

If $g_i=\{[a_{p'},b_{q'}]\}$ is a new edge as in Op 2 and $g_j=\{[a_{p},b_{q}]\} $
is a new edge as in Op 1,
 then we get a contradiction in a similar manner.

\vskip 2mm \noindent
\textsc{Case 5:} 
 Suppose $g_i=\{[a_{p'},b_{q'}]\}$ is a new edge as in Op 1 for some 
$1 \leq p' \leq \alpha$, $1 \leq q' \leq \beta'$ and $g_j$ is a new 
edge as in Op 3. Note that
$g_j=\{a,a_\mu\}$ for some $1 \leq \mu \leq \alpha'$
or $g_j=\{b, b_{\mu}\}$ for some $1 \leq \mu \leq \beta'$
or $g_j=\{[a_p, b_q]\}$  for some $1 \leq p \leq \alpha'$, $1 \leq q \leq \beta' $.
Suppose $g_j=\{a,a_\mu\}$ for some $1 \leq \mu \leq \alpha'$. By Op 1, we have
\[
 \{a,a_{p'}\}<g_i< g_j=\{a,a_{\mu}\}. 
\]
This is a contradiction to $\{g_i,g_j\}$ forms an induced $2K_2$-subgraph. Suppose
$g_j=\{b, b_{\mu}\}$ for some $1 \leq \mu \leq \beta'$. Since $g_i$ is a new edge as in Op 1,
we have $$\{a,a_{p'}\}<\{[a_{p'},b_{1}]\}<\cdots<\{[a_{p'},b_{\mu}]\}<\cdots<\{[a_{p'},b_{\beta'}]\}.$$
Therefore $\{[a_{p'},b_{\mu}]\}<g_{j}$. This is a contradiction to $\{g_i,g_j\}$ forms an induced $2K_2$-subgraph.
Suppose $g_j=\{[a_p, b_q]\}$  for some $1 \leq p \leq \alpha'$, $1 \leq q \leq \beta' $.
It can also seen that $\{[a_{p'},b_{q}]\}<g_{j}$. This is a contradiction to $\{g_i,g_j\}$ forms an induced $2K_2$-subgraph.

If $g_i$ is a new edge as in Op 3 and $g_j$
is a new edge as in Op 1,
 then we get a contradiction in a similar manner.

 \vskip 2mm \noindent
\textsc{Case 6:}
 Suppose either $g_i$ is a new edge as in Op 2 and $g_j$ is a new edge as in Op 3
or $g_j$ is a new edge as in Op 3 and $g_i$
is a new edge as in Op 2.
Proceeding as in the \textsc{Case 5}, one can show that 
$g_i$ and $g_j$ can not form an induced $2K_2$-subgraph.

In all cases we get a contradiction to the assumption that
$\mathcal{K}_j$ has an induced $2K_2$-subgraph for some $1 \leq j \leq t_{m'}$. Therefore
$\mathcal{K}_j$ has no induced $2K_2$-subgraph for all $1 \leq j \leq t_{m'}$.
By Lemma \ref{old-new}, $(V(H_m'),\{\mathfrak{g}_1,\ldots,\mathfrak{g}_{r'}\})$
has no induced $2K_2$-subgraph for all $1 \leq r' \leq t_{m'}$.
By Lemma \ref{main-lemmatech}, $H_m'$ is a co-chordal graph. Therefore, $H_m'$ is a co-chordal graph 
for all $1 \leq m \leq \widetilde{n}$.
Hence $\cochord(\PP) \leq \widetilde{n}$.
\end{proof} 

 As a consequence of Lemma \ref{cochord-lemma} one has:

\begin{corollary}\label{rs-cor}
 Let $I$ and $J$ be edge ideals with $I \subseteq J$. If $J$ has a linear minimal free resolution
 and for any $ab \in I$, then 
 $(IJ:ab)$ also has a linear minimal free resolution
\end{corollary}
\begin{proof}
 Let $G$ and $\PP$ be the graphs associated to $J$ and $\widetilde{(IJ:ab)}$ respectively. By
 \cite[Theorem 1]{froberg}, $G$ is a co-chordal graph and by Lemma \ref{cochord-lemma},
 $\PP$ is also co-chordal. Again by \cite[Theorem 1]{froberg}, $\PP$ has a linear minimal free 
 resolution. Therefore, $(IJ:ab)$ has a linear minimal free resolution.
\end{proof}

\section{Upper and lower bound for the regularity of product of two edge ideals}\label{2-edgeideals}
In this section, we obtain a general upper and lower bounds for the regularity of product of two edge ideals.

We start by recalling the notion of 
\textit{upper-Koszul simplicial complexes} associated to monomial ideals. 
 Let $I \subseteq R = \K[x_1,\ldots,x_n]$ be a monomial ideal and 
 let $\alpha=(\alpha_1,\ldots,\alpha_n) \in \mathbb{N}^n$ 
 be a $\mathbb{N}^n$-graded degree. The \textit{upper-Koszul simplicial complex}
 associated to $I$ at degree $\alpha$, 
 denoted by $K^{\alpha}(I)$, is the simplicial complex over 
 $V = \{x_1,\ldots,x_n\}$ whose faces are:
 \[
  \Big\{W \subseteq V \mid \frac{x_1^{\alpha_1} \cdots x_n^{\alpha_n}}{\prod\limits_
  {u \in W}u}
  \in I\Big\}.
 \]

Given a monomial ideal $I$, its $\mathbb{N}^n$-graded Betti numbers are given by the following
formula of Hochster (\cite[Theorem 1.34]{ms2005})
\[
 \beta_{i,\alpha}(I)=\dim_\K \widetilde{H}_{i-1}(K^\alpha(I);\K) \text{ for all }
i \geq 0 \text{ and } \alpha \in \mathbb{N}^n.
 \]

Now,  we  prove the general lower bound for the regularity of product of edge ideals.
One can see that Beyarslan et al., proof of Lemma 4.2 in \cite{selvi_ha}
works more generally and we generalize their argument to prove it below:
\begin{theorem}\label{main-lower} 
Let $J_i=I(G_i)$ be the edge ideal of $G_i$ and $J_1 \subseteq \cdots \subseteq J_d$ 
for all $1 \leq i \leq d$.
  Then $$2d+\nu_{G_1 \cdots G_d}-1 \leq \reg(J_1 \cdots J_d).$$
 
\end{theorem}
\begin{proof}
 Let $f_1,f_2,\ldots,f_{\nu_{G_1 \cdots G_d}}$ be the induced matching of $G_i$ 
 for all $1 \leq i \leq d$. 
 Let
 $Q$ be an induced subgraph of $G_i$  with $E(Q)=\{f_1,\ldots,f_{\nu_{G_1 \cdots G_d}}\}$
 for all $1\leq i \leq d$.
First, we claim that if for any 
$\alpha=(\alpha_1,\ldots,\alpha_n) \in \mathbb{N}^n$ and 
$\supp(\alpha) \subseteq V(Q)$, where $\supp(\alpha)=\{x_i \mid \alpha_i \neq 0\}$, then 
$K^{\alpha}(I(Q)^d)=K^{\alpha}(J_1 \cdots J_d)$.
Clearly, $K^{\alpha}(I(Q)^d) \subseteq K^{\alpha}(J_1 \cdots J_d)$. 
Suppose $W \in K^{\alpha}(J_1 \cdots J_d)$.
Since $\supp(\alpha) \subseteq V(Q)$, we have $W \subseteq V(Q)$.
Then  $m=\frac{x_1^{\alpha_1} 
\cdots x_n^{\alpha_n}}{\prod\limits_{u \in W}u}  \in J_1 \cdots J_d$, which
implies that $g_1 \cdots g_d \mid m$ where $g_i \in J_i$ for all $1\leq i \leq d$.
Clearly $\supp(g_i) \subseteq \supp(m)$ for all $1 \leq i \leq d$.
Therefore $g_i \in I(Q)$ for all $1 \leq i \leq d$. Then 
$m=\frac{x_1^{\alpha_1} 
\cdots x_n^{\alpha_n}}{\prod\limits_{u \in W}u}  \in I(Q)^d$, which
implies that $W \in K^{\alpha}(I(Q)^d)$. Hence the claim.
It follows from \cite[Theorem 1.34]{ms2005} that 
\[
 \beta_{i,\alpha}(I(Q)^d)=\dim_\K \widetilde{H}_{i-1}(K^\alpha(I(Q)^d);\K)=
 \dim_\K \widetilde{H}_{i-1}(K^\alpha(J_1 \cdots J_d);\K)=\beta_{i,\alpha}(J_1 \cdots J_d).
\]
Therefore,
\begin{align*}
 \beta_{i,j}(I(Q)^d)&=\sum\limits_{\alpha \in \mathbb{N}^n,
 ~\supp(\alpha) \subseteq V(Q),~|\alpha|=j}\beta_{i,\alpha}(I(Q)^d)\\
 &=\sum\limits_{\alpha \in \mathbb{N}^n,
 ~\supp(\alpha) \subseteq V(Q),~|\alpha|=j}\beta_{i,\alpha}(J_1 \cdots J_d)\\
 &\leq \sum\limits_{\alpha \in \mathbb{N}^n,
 ~|\alpha|=j}\beta_{i,\alpha}(J_1 \cdots J_d)=\beta_{i,j}(J_1 \cdots J_d).
\end{align*}
Hence $\reg(I(Q)^d) \leq \reg(J_1 \cdots J_d)$.
 By \cite[Lemma 4.4]{selvi_ha},
$  \reg(I(Q)^d)=2d+\nu_{G_1 \cdots G_d}-1.$ Hence $2d+\nu_{G_1 \cdots G_d}-1 \leq \reg(J_1 \cdots J_d).$
\end{proof}

We now prove an upper bound for the regularity of $IJ$.

\begin{theorem}\label{main}
Let $I$ and $J$ be as in Set-up \ref{setup}.
 Then 
\begin{align}\label{eq:main}
 \reg(IJ) \leq \max \{\cochord(G)+3,~\reg(I)\}.
\end{align}
In particular,
\[
  \reg(IJ) \leq \max \{\cochord(G)+3,~\cochord(H)+1\}.
 \]
\end{theorem}
\begin{proof}
Set $I=(f_1,\ldots,f_t)$. It follows
from set of short exact sequences:
 \begin{eqnarray}\label{main-exact-seq1}
  0 & \longrightarrow & \frac{R}{(IJ : f_1)}(-2)
  \overset{\cdot f_1}{\longrightarrow} \frac{R}{IJ} \longrightarrow
  \frac{R}{(IJ , f_1)} \longrightarrow 0;\nonumber \\
& & \hspace*{0.5cm} \vdots\hspace*{2.5cm}\vdots \hspace*{3cm}\vdots \\ 
0 & \longrightarrow & \frac{R}{((IJ,
  f_1,\ldots,f_{t-1}):f_t)}(-2) 
\overset{\cdot f_t}{\longrightarrow}
\frac{R}{(IJ, f_1,\ldots,f_{t-1})} \longrightarrow
\frac{R}{(IJ, I)} \longrightarrow 0, \nonumber
\end{eqnarray} 
 that
\begin{equation*}\label{main:eq} 
  \reg\left(\frac{R}{IJ}\right)  \leq  \max \left\{
	\begin{array}{l}
	  \reg\left(\frac{R}{(IJ :f_1)}\right)+2,\ldots, 
		\reg\left(\frac{R}{(IJ,f_1,\ldots,f_{t-1}):f_t)}\right) + 2,~ 
		\reg\left(\frac{R}{I}\right)
	  \end{array}\right\}.
\end{equation*}
Note that $((IJ,f_1,\ldots,f_{i-1}):f_i)=(IJ:f_i)+(\text{variables})$ for any $1 \leq i \leq t$.
By \cite[Theorem 1.2]{KalaiMes} and Corollary 
   \ref{pol_reg}, we have 
   $\reg((IJ,f_1,\ldots,f_{i-1}):f_i)  \leq \reg((IJ:f_i))
   = \reg(\widetilde{(IJ:f_i)}).$
Let $\PP_i$ be the graph associated to $\widetilde{(IJ:f_i)}$. Therefore, by \cite[Theorem 1]{russ} and 
Lemma \ref{cochord-lemma},
$\reg(\widetilde{IJ:f_i}) \leq \cochord(\PP_i)+1 \leq \cochord(G)+1.$
Hence $\reg(IJ) \leq \max \{\cochord(G)+3, ~\reg(I)\}$.
Now the second assertion follows  from \cite[Theorem 1]{russ}.
 \end{proof}
\begin{remark}
 Let $G$ be a graph and $H$ be a subgraph of $G$. We would like to note here that
the invariant $\cochord(G)$ and $\cochord(H)$ are not comparable in general.  
For example, if $G$ is the graph with 
$E(G)=\{\{x_1,x_2\}$, $\{x_2,x_3\}$, $\{x_3,x_4\}$, $\{x_4,x_5\}$, $\{x_5,x_1\}$, $\{x_1,x_3\}\}$
and $H$ is a subgraph of $G$ with $E(H)=\{\{x_1,x_2\}$, $\{x_2,x_3\}$, $\{x_3,x_4\}$, $\{x_4,x_5\}$,
$\{x_5,x_1\}\}$, 
then $\cochord(G)=1$ and $\cochord(H)=2$. If $G$ is a graph with $E(G)=\{\{x_1,x_2\}$, $\{x_2,x_3\}$,
$\{x_3,x_4\}$,$\{x_4,x_5\}\}$ and $H$ is a graph with 
$E(H)=\{\{x_1,x_2\},\{x_2,x_3\}\}$, then $\cochord(G)=2$ and $\cochord(H)=1$.
\end{remark}

As an immediate consequence, we have the following statements.
\begin{corollary}\label{mat-bound} 
Let $I$ and $J$ be as in Set-up \ref{setup}.
 Then $\reg(IJ) \leq \ma(G)+3.$
\end{corollary}
\begin{proof}
 Since $H$ is a subgraph of $G$, $\ma(H) \leq \ma(G)$. Hence
the assertion follows from Theorem \ref{main}.
\end{proof}
The following example shows that the inequalities given in 
Theorem \ref{main-lower} and 
 Corollary \ref{mat-bound} are sharp.
\begin{example}
 Let $I(H)=(x_2x_3,x_4x_5)$ and $I(G)=(x_1x_2,x_1x_3,x_1x_4,x_1x_5,x_2x_3,x_4x_5)$ be the 
 edge ideals.
 It is not hard to verify that $\ma(G)=2$ and $\nu_{HG}=2$. Therefore, by 
 Theorem \ref{main-lower} and 
 Corollary \ref{mat-bound},
 $\reg(I(H)I(G))=5$.
\end{example}

\begin{corollary} \label{ind-eql} 
 Let $I$ and $J$ be as in Set-up \ref{setup}.
 If $H$ is an induced subgraph of $G$, then 
 \[
  \nu(H)+3 \leq \reg(IJ) \leq \cochord(G)+3.
 \]
 \end{corollary}
 \begin{proof}
 If $H$ is an induced subgraph of $G$, then $\cochord(H) \leq \cochord(G)$ and $\nu_{HG}=\nu(H)$.
 Therefore, by Theorem \ref{main-lower}, Theorem \ref{main},
 $\nu(H)+3 \leq \reg(IJ) \leq \cochord(G)+3$.
 \end{proof}

\vskip 2mm \noindent
It follows from Corollary \ref{mat-bound} that if 
$G_1$ is a subgraph of $G_2$, then 
\[
 \reg(J_1J_2) \leq 3+\ma(G_2),
\]
where $J_i=I(G_i)$ for all $i=1,2$.
As a natural extension of this result, one tend to think that the same expression 
may hold true for $\reg(J_1 \cdots J_d)$. 
Also, this question is inspired by previous work of the regularity
of powers of edge ideals of graphs (\cite{BBH17}, \cite{jayanthan}, \cite{JS21}). 
More precisely,
we would like to ask:
\begin{question}
 If $G_{i-1}$ is a subgraph of $G_i$ for all $i=2,\ldots,d$,
 is it true that $$\reg(J_1 \cdots J_d) \leq 2d+\ma(G_d)-1,$$ where
 $J_i=I(G_i)$ for all $1 \leq i \leq d$?
 In particular, if $G_{i-1}$ is an induced subgraph of $G_i$ for all $i=2,\ldots,d$,
 is it true that $$\reg(J_1 \cdots J_d) \leq 2d+\cochord(G_d)-1?$$
\end{question}

The following example shows that the above inequality can be equality.

\begin{example}
 Let $J_1=(\{x_{i-1}x_i \mid 5 \leq i \leq 6\})$,
 $J_2=J_3=(\{x_{i-1}x_i \mid 3 \leq i \leq 8\})$ and 
 $J_4=J_5=(\{x_{i-1}x_i \mid 2 \leq i \leq 10\})$ be the edge ideals. 
 Set $J_i=I(G_i)$ for all $1 \leq i \leq 5$.
 A computation on \textsc{Macaulay2} shows that $\reg(J_1 \cdots J_5)=12$.
 Note that $G_{i-1}$ is an induced subgraph of $G_{i}$ for all $2 \leq i \leq 5$ and 
 $\cochord(G_5)=3$. Then $\reg(J_1 \cdots J_5)=12 \leq 2.5+\cochord(G_5)-1=12$.
\end{example}

Let $G_1$ and $G_2$ be graphs with disjoint vertex sets 
(i.e., $V(G_1) \cap V(G_2) = \emptyset$).
The join of $G_1$ and
$G_2$, denoted by $G_1*G_2$, is the graph on the vertex set $V(G_1) \cup V(G_2)$ whose edge set
is $E(G_1 * G_2) = E(G_1) \cup E(G_2) \cup \{\{x, y\} \mid x \in V(G_1) \text{ and } y \in V (G_2)\}.$

\begin{corollary}
 Let $G_1$, $G_2$ be graphs with disjoint edges and $G=G_1*G_2$. If $H=G_1$ or $H=G_2$, then
 \[
  \nu(H)+3 \leq \reg(I(H)I(G)) \leq \max \{\cochord(G_1), \cochord(G_2)\}+3.
 \]
 In particular, if $\cochord(G_1) \leq \cochord(G_2)$ and $H=G_2$, then $\reg(I(H)I(G))=\nu(G_2)+3$.
\end{corollary}
\begin{proof}
 If $H$ is equal to either $G_1$ or $G_2$, then $H$ is an induced subgraph of $G$. 
 Therefore, by Corollary \ref{ind-eql}, $\nu(H)+3 \leq \reg(I(H)I(G)) \leq \max\{\cochord(G)+3,\cochord(H)+1\}$. 
 By \cite[Proposition 4.12]{selva}, $\cochord(G)=\max\{\cochord(G_1),~\cochord(G_2)\}$.
 Therefore $\reg(I(H)I(G)) \leq \max \{\cochord(G_1), \cochord(G_2)\}+3.$
\end{proof}

\section{Precise expressions for the regularity of product of edge ideals}\label{precise}
 
In this section, we explicitly compute the regularity of product of edge ideals for certain
classes of graphs. First, we compute the regularity of $IJ$
when $J$ has linear resolution.

\begin{theorem}\label{regularity} 
Let $I$ and $J$ be edge ideals with $I \subseteq J$.
Suppose $J$ has linear resolution.
\begin{enumerate}
 \item If $\reg(I) \leq 4$, then $IJ$ has linear resolution.
 \item If $5 \leq \reg(I)$, then 
 $\reg(IJ)=\reg(I)$.
\end{enumerate}
\end{theorem}
\begin{proof}
 Suppose $\reg(I) \leq 4$. Since $J$ has linear resolution, by \eqref{eq:main},
  $4 \leq \reg(IJ) \leq \max\{4, ~\reg(I)\}.$
Hence $\reg(IJ)=4$.

Suppose $\reg(I) \geq 5$. By \eqref{eq:main}, we have
 $\reg(IJ) \leq \max\{4, ~\reg(I)\} \leq \reg(I).$
Since $4 \leq \reg(R/I)$, there exist $i,j$ such that $j-i \geq 4$ and $\beta_{i,j}(R/I) \neq 0$.
From Equation \eqref{main-exact-seq1}, either $\beta_{i,j}\left(\dfrac{R}{(IJ, f_1,\ldots,f_{t-1})}\right) \neq 0$ or
$\beta_{i-1,j} \left( \dfrac{R}{(( IJ, f_1,\ldots,f_{t-1}) : f_t)}(-2) \right) \neq 0$. 
Note that
$((IJ,f_1,\ldots,f_{t-1}):f_t)=(IJ:f_t)+(\text{variables}).$ 
Since $J$ has linear resolution, by  Corollary \ref{rs-cor}, 
$(IJ:f_t)$ has linear resolution. Hence 
$(IJ,f_1,\ldots,f_{t-1}):f_t)$ has linear resolution i.e
$\reg((IJ,f_1,\ldots,f_{t-1}):f_t))=2$.
If $\beta_{i-1,j-2} \left( \dfrac{R}{(( IJ, f_1,\ldots,f_{t-1}) : f_t)} \right) \neq 0$, then
$\reg \left( \dfrac{R}{(( IJ, f_1,\ldots,f_{t-1}) : f_t)} \right) \geq j-1-i \geq 4-1=3.$
This is a contradiction to 
$\reg \left( \dfrac{R}{(( IJ, f_1,\ldots,f_{t-1}) : f_t)} \right) \leq 1.$
Therefore, $\beta_{i,j}\left(\dfrac{R}{(IJ, f_1,\ldots,f_{t-1})}\right) \neq 0.$
Then again either $\beta_{i,j}\left(\dfrac{R}{(IJ, f_1,\ldots,f_{t-2})}\right) \neq 0 $ or
$\beta_{i-1,j} \left( \dfrac{R}{(( IJ, f_1,\ldots,f_{t-2}) : f_{t-1})}(-2) \right) \neq 0.$
As in the previous case, we get $\beta_{i,j}\left(\dfrac{R}{(IJ, f_1,\ldots,f_{t-2})}\right) \neq 0$.
Then one proceeds in the same manner. At each stage, we get either
$\beta_{i,j}\left(\dfrac{R}{(IJ, f_1,\ldots,f_{l-1})}\right) \neq 0$ or
$\beta_{i-1,j} \left( \dfrac{R}{(( IJ, f_1,\ldots,f_{l-1}) : f_l)}(-2) \right) \neq 0$  for all  $l$.
Therefore, $\beta_{i,j}\left(\dfrac{R}{IJ}\right) \neq 0$. Hence $ \reg(R/I) \leq \reg(R/IJ).$
\end{proof}

One of the immediate consequence of Theorem \ref{regularity} is the following:

\begin{corollary}
 Let $I$ and $J$ be as in Set-up \ref{setup}.  If 
 $J$ has linear resolution and  $\nu(H) \geq 4$, then $\reg(IJ)=\reg(I)$. In particular,
 \[
  \nu(H)+1 \leq \reg(IJ) \leq \cochord(H)+1.
 \]
\end{corollary}
\begin{proof}
 By \eqref{ag_reg_chg}, $5 \leq \reg(I)$. Therefore, by Theorem \ref{regularity},
 $\reg(IJ)=\reg(I)$. The second assertion follows from \eqref{ag_reg_chg}.
\end{proof}

A graph which is isomorphic to the graph with vertices $a, b, c, d$ and edges 
$\{a, b\}$, $\{b, c\}$, $\{a,c\},$
$\{a, d\}$, $\{c, d\}$
is called a \textit{diamond}. A graph which is isomorphic to the graph with vertices 
$w_1, w_2, w_3, w_4, w_5$ and edges $\{w_1,w_3\}, \{w_2,w_3\}, \{w_3,w_4\}, \{w_3,w_5\}, \{w_4,w_5\}$ 
is called a \textit{cricket}.
A graph without an induced diamond (cricket) is called 
diamond (cricket) -free.

\begin{corollary}\label{se-cl} Let $I$ and $J$ be as in Set-up \ref{setup}. Suppose $J$ has linear resolution. Then $IJ$ has linear
resolution if 
 \begin{enumerate}
  \item $\cochord(H) \leq 3$;
  \item $H$ is  (gap,cricket)-free;
  \item $H$ is (gap, diamond)-free;
  \item $H$ is (gap, $C_4$)-free~~~~~or
  \item $H$ is a graph such that $H^c$ has no triangle;
 \end{enumerate}
\end{corollary}

\begin{proof}
 By \eqref{ag_reg_chg},  \cite[Theorem 3.4]{banerjee},
 \cite[Theorem 3.5]{Nur18}, \cite[Proposition 2.11]{Nursel} and \cite[Theorem 2.10]{MorKiani},
 $\reg(I) \leq 4$. Therefore, by Theorem \ref{regularity}, $IJ$ has linear resolution.
\end{proof}

So far, we had been discussing about the regularity of product of two edge ideals. Now we study
the regularity of product of more than two edge ideals.
\begin{theorem} \label{prd-ideals}
 Let $J_1,\ldots,J_d$ be edge ideals and $J_1 \subseteq J_2 \subseteq 
 \cdots \subseteq J_d$, $d \in \{3,4\}$. Suppose $J_d$ is the edge ideal of a complete graph.
 \begin{enumerate}
  \item If $\reg(J_1\cdots J_{d-1}) \leq 2d$, then $J_1 \cdots J_d$ has linear resolution.
  \item If $\reg(J_1\cdots J_{d-1}) \geq 2d+1$, then 
   $\reg(J_1 \cdots J_d)=\reg(J_1 \cdots J_{d-1})$.
 \end{enumerate}
\end{theorem}
\begin{proof}
Set $\mathcal{J}:=J_1 \cdots J_d$ and $J_1\cdots J_{d-1}=(\FA_1,\ldots,\FA_t)$.
Now we claim that, if 
$(\FA_j:\FA_i)=(u^s)$ for some $s \geq 3$ and $j \neq i$, then $u^2 \in 
(\mathcal{J}:\FA_i)$. 
Clearly $d>3$.
 Set $\FA_j=g_1g_2g_3$ and $\FA_i=f_1f_2f_3$, where $g_i,f_i \in J_i$
 for all $1 \leq i \leq 3$. Since $s \geq 3$, we have $u \mid g_i$ and $u \nmid f_i$ for all $1 \leq i \leq 3$.
 Set $g_1=ua$, $g_2=ub$, $g_3=uc$, $f_1=x_1x_2$, $f_2=x_3x_4$ and  $f_3=x_5x_6$ ($x_i$ may be equal to $x_j$, for some $1 \leq i,j \leq 5$). Note that $abc \mid f_1f_2f_3$.
 If $ab \mid f_i$ and $c \mid f_j$, for some $1 \leq i,j \leq 3$, 
 then $uaubf_jf_k \in \mathcal{J}$, where $k \neq i,j$.
  If $a \mid f_i$, $b \mid f_j$, $c \mid f_k$ for some 
 $1 \leq i,j,k \leq 3$, then $uaub f_k (\frac{f_if_j}{ab}) \in \mathcal{J}$. 
 Therefore $u^2 \in (\mathcal{J}:\FA_i)$. Hence the claim.

 Let $m \in \mathcal{G}(\mathcal{J}:\FA_i).$ By degree consideration $m$ can not have degree 1.
We now claim that $\deg(m)=2$. 
  Suppose $|\supp(m)| \geq 2$.
Since $J_d$ is a edge ideal of complete graph, $\deg(m)=2$.
Suppose $|\supp(m)|=1$. Assume that $\deg(m) \geq 3$. Set $m=u^s$ for some 
$s \geq 3$. Clearly $n_1 \cdots n_d \mid u^s\FA_i$,
where $n_l \in \mathcal{G}(J_l)$ for all $1 \leq l \leq d$.
Then $n_1 \cdots n_{d-1} \mid u^s\FA_i$. Also, 
$u^s \in (n_1 \cdots n_{d-1}:\FA_i)$. By above claim, $u^2 \in (\mathcal{J}:\FA_i)$.
This is contradiction to $\deg(m) \geq 3$. Therefore
$\deg(m)=2$.

 By the above arguments, one can see that the ideal $((\mathcal{J},\FA_1,\ldots,\FA_{i-1}):\FA_i)$ is 
 generated by quadratic monomial ideals. 
 Note that $J_d \subseteq (\mathcal{J}:\FA_i)$. 
 Let $K_i$ be the graph associated
 to $\widetilde{((\mathcal{J},\FA_1,\ldots,\FA_{i-1}):\FA_i)}$.
 Since $J_d$ is the edge ideal of complete graph, $K_i$ is the graph obtained from
 complete graph by attaching pendant to  some vertices. Hence 
 $K_i$ is a co-chordal graph. 
 By \cite[Theorem 1]{froberg}, 
 $\reg((\mathcal{J},\FA_1,\ldots,\FA_{i-1}):\FA_i))=2$ for all $1 \leq i \leq t$.
 
 Consider the similar exact sequences as in \eqref{main-exact-seq1}, we get 
 \begin{equation*}\label{main:eq} 
  \reg\left(\frac{R}{\mathcal{J}}\right)  \leq  \max \left\{ 
	\begin{array}{l}
	  \reg\left(\frac{R}{(\mathcal{J} :\FA_1)}\right)+2(d-1),\ldots, 
		\reg\left(\frac{R}{(\mathcal{J},\FA_1,\ldots,\FA_{t-1}):\FA_t)}
		\right) + 2(d-1),~ \\
		\reg\left(\frac{R}{J_1 \cdots J_{d-1}}\right)
	  \end{array}\right\}.
\end{equation*}
Therefore
 $\reg\left(\frac{R}{\mathcal{J}}\right)  \leq  \max \left\{ 
	\begin{array}{l}
	  2d, ~\reg\left(\frac{R}{J_1 \cdots J_{d-1}}\right)
	  \end{array}\right\}.$
Proceeding as in the proof of Theorem \ref{regularity} we will get the desired conclusion. 
\end{proof}

As an immediate consequence of Theorem \ref{main}, Theorem \ref{prd-ideals}, we obtain an upper bound for the 
regularity of product of edge ideals in terms of co-chordal cover numbers.
\begin{corollary}\label{prd-us}
 Let $J_i=I(G_i)$ be the edge ideal of $G_i$ for all $1 \leq i \leq d$ and $J_1 \subseteq 
 \cdots \subseteq J_d$.
\begin{enumerate}
 \item If $G_3$ is a complete graph, then 
 \[
  \reg(J_1J_2J_3) \leq \max\{6,~\cochord(G_2)+3,~\cochord(G_1)+1\}.
 \]
\item If $G_i$ is a complete graph for all $i=3,4$, then 
 \[
  \reg(J_1J_2J_3J_4) \leq \max\{8,~\cochord(G_2)+3,~\cochord(G_1)+1\}.
 \]
\end{enumerate}
\end{corollary}
As a consequence of Theorem \ref{prd-ideals}, we give sufficient conditions 
for product of edge ideals to have linear resolutions.

\begin{corollary}\label{ls-pr}
 Let $J_i=I(G_i)$ be the edge ideal of $G_i$ for all $1 \leq i \leq d$ and $J_1 \subseteq 
 \cdots \subseteq J_d$. 
 \begin{enumerate}
  \item If $G_3$ is a complete graph and $ \max\{\cochord(G_2)+3, \cochord(G_1)+1\} \leq 6$, then $J_1J_2J_3$ has 
  linear resolution.
  \item If $G_i$ is a complete graph for all $i=3,4$ and 
  $\max\{\cochord(G_2)+3,~\cochord(G_1)+1\} \leq 8$, then 
  $J_1J_2J_3J_4$ has   linear resolution.
  \item If $G_4$ is a complete graph and $G_i$ is an induced subgraph of $G_{i+1}$ for all $1 \leq i \leq 3$, then 
  $J_1J_2J_3J_4$ has   linear resolution.
  \item If $G_i$ is a complete graph for all $i=3,4$ and $J_1J_2$ has linear resolution, then
  $J_1J_2J_3J_4$ has linear resolution.
 \end{enumerate}
\end{corollary}
\begin{proof}
 (1) and (2): The assertions follow from Theorem \ref{main} and Theorem \ref{prd-ideals}.
 \vskip 0.5mm \noindent
 (3) Since $G_4$ is a complete graph and $G_i$ is an induced subgraph of $G_{i+1}$ for all $1 \leq i \leq 3$,
 $G_i$ is a complete graph for all $1 \leq i \leq 3$. Therefore, by Corollary \ref{prd-us}(1),
 $J_1J_2J_3$ has linear resolution. Hence, by Theorem \ref{prd-ideals}, $J_1J_2J_3J_4$ has linear resolution.
\vskip 1mm \noindent
(4) If $J_1J_2$ has linear resolution, then by Theorem \ref{prd-ideals}, $J_1J_2J_3$ has linear
resolution. Therefore, by Theorem \ref{prd-ideals}, $J_1J_2J_3J_4$ has linear resolution.
\end{proof}

\vskip 2mm
\noindent
\textbf{Acknowledgement:} 
The computational commutative algebra
package Macaulay2\cite{M2} was heavily used to compute several
examples.
The third author is partially supported by DST, Govt of India under
the DST-INSPIRE (DST/Inspire/04/2019/001353) Faculty Scheme.
We would also like to express our sincere gratitude
to anonymous referee for meticulous reading and suggesting several improvements.

\bibliographystyle{abbrv}
\bibliography{refs_reg} 
\end{document}